\documentclass[12pt]{amsart}

\usepackage[utf8]{inputenc}
\usepackage[T1]{fontenc}

\usepackage{amssymb}
\usepackage{amsmath}
\usepackage{amsfonts}
\usepackage{latexsym}
\usepackage{amsthm}
\usepackage{srcltx}

\usepackage{enumerate}
\newtheorem{Theorem}{Theorem}

\newtheorem{propo}[Theorem]{Proposition}

\newtheorem{Definition}[Theorem]{Definition}

\newtheorem{rema}[Theorem]{Remark}

\def\F{\mathbb{F}}

\newcommand{\ds}{\displaystyle}

\newcommand{\mus}{\mu^\mathrm{sym}}

\newcommand{\Ms}{M^\mathrm{sym}}
\newcommand{\ms}{m^\mathrm{sym}}

\begin{document}

\title[Modular curves and uniform tensor rank]{Dense families of modular curves, prime numbers and uniform symmetric tensor rank of multiplication in certain finite fields}

\author[S. Ballet]{St\'ephane Ballet}
\address{
Stéphane Ballet
\newline \indent 
Aix Marseille Univ, {CNRS}, Centrale Marseille, {I2M}, Marseille, 
France. Institut de
Math\'ematiques de Marseille, Case 907, 163 Avenue de Luminy, F-13288
Marseille Cedex 9, France.}
\email{stephane.ballet@univ-amu.fr}

\author[A. Zykin]{Alexey Zykin}
\address{
Alexey Zykin
\newline \indent
Laboratoire GAATI, Universit\'e de la Polyn\'esie fran\c caise
\newline \indent
BP 6570 --- 98702 Faa'a, Tahiti, Polyn\'esie fran\c caise
\newline \indent
National Research University Higher School of Economics
\newline \indent
AG Laboratory NRU HSE 
\newline \indent
Institute for Information Transmission Problems of the Russian Academy of Sciences
}
\email{alzykin@gmail.com}

\thanks{The second author was partially supported by ANR Globes ANR-12-JS01-0007-01 and by the Russian Academic Excellence Project '5-100'. }

\date{\today}

\keywords{Algebraic function field, tower of function fields, tensor rank, algorithm, finite field, modular curve, Shimura curve}

\begin{abstract}
We obtain new uniform bounds for the symmetric tensor rank of multiplication in finite extensions of any finite field $\F_p$ or $\F_{p^2}$ 
where $p$ denotes a prime number $\geq 5$. 
In this aim, we use the symmetric Chudnovsky-type generalized algorithm applied on sufficiently dense families of modular 
curves defined over $\F_{p^2}$ attaining the Drinfeld--Vladuts bound and on the descent of these families to the definition field $\F_p$. 
These families are obtained thanks to prime number density theorems of type Hoheisel, in particular a result due to Dudek (2016).
\end{abstract}

\maketitle



\section{Introduction}

\subsection{Notation}

Let $q=p^s$ be a prime power, $\F_q$ be the finite field with $q$ elements and $\F_{q^n}$
be the degree $n$ extension of $\F_q$. The multiplication of two elements of $\F_{q^n}$ is 
an $\F_q$-bilinear application from $\F_{q^n} \times \F_{q^n}$ onto $\F_{q^n}$.
It can be considered as an $\F_q$-linear application from the tensor product 
${\F_{q^n} \otimes_{\F_q} \F_{q^n}}$
onto $\F_{q^n}$. Consequently it can be also viewed as an element 
$T$ of ${(\F_{q^n} \otimes_{\F_q} \F_{q^n})^\star \otimes_{\F_q} \F_{q^n}}$,
namely an element of ${\F_{q^n}^{\,\star} \otimes_{\F_q} \F_{q^n}^{\,\star} \otimes_{\F_q} \F_{q^n}}$.
More precisely, when $T$ is written
\begin{equation}\label{tensor}
T=\sum_{i=1}^{r} x_i^\star\otimes y_i^\star\otimes c_i,
\end{equation}
where the $r$ elements $x_i^\star$ and the $r$ elements $y_i^\star$
are in the dual $\F_{q^n}^{\,\star}$ of $\F_q$ and the $r$ elements $c_i$ are in $\F_{q^n}$,
the following holds for any ${x,y \in \F_{q^n}}$:
$$
x\cdot y=\sum_{i=1}^r x_i^\star(x) y_i^\star(y) c_i.
$$

\begin{Definition}
The minimal number of summands in a decomposition of the multiplication tensor $T$ 
is called the rank of the tensor of the multiplication in the extension field $\F_{q^n}$ (or bilinear complexity of the multiplication) 
and is denoted by
$\mu_{q}(n)$:
$$
\mu_{q}(n)= \min\left\{r \; \Big| \; T=\sum_{i=1}^{r} x_i^\star\otimes y_i^\star\otimes c_i\right\}.
$$
\end{Definition}

It is known that the tensor $T$ can have a symmetric decomposition:
\begin{equation}\label{symtensor}
T=\sum_{i=1}^{r} x_i^\star\otimes x_i^\star\otimes c_i.
\end{equation}
 
\begin{Definition}
The minimal number of summands in a symmetric decomposition of the multiplication tensor $T$ 
is called the symmetric tensor rank of the multiplication (or the symmetric bilinear complexity of the multiplication) and is denoted by
$\mus_{q}(n)$:
$$
\mus_{q}(n)= \min\left\{r \; \Big| \; T=\sum_{i=1}^{r} x_i^\star\otimes x_i^\star\otimes c_i\right\}.
$$
\end{Definition}

From an asymptotical point of view, let us define the following \begin{equation}\label{M1}
\Ms_q= \limsup_{k \rightarrow \infty}\frac{\mus_q(k)}{k}, 
\end{equation}
\begin{equation}\label{m1}
\ms_q=\liminf_{k \rightarrow \infty}\frac{\mus_q(k)}{k}. 
\end{equation}

Let $F/\F_q$ be a function field of genus $g$ over the finite field $\F_q$ and $N_k(F)$ be the number
of places of degree $k$ of $F/\F_q$.

Let us define:
$$
N_q(g)= \max \big\{N_1(F)\, |\, F \mbox{ is a function field over }\F_q \mbox{ of genus }g  \big\}
$$
and
$$
A(q)=\limsup_{g\rightarrow +\infty} \frac{N_q(g)}{g}.
$$
We know that (Drinfeld--Vladuts bound):
$$A(q) \leq q^{\frac{1}{2}}-1,$$
the bound being attained if $q$ is a square.

\subsection{Known results}

The original algorithm of D.V. and G.V. Chudnovsky
introduced in \cite{chch} is symmetric by definition and leads to the
two following results from \cite{ball1}, \cite{baro1} and \cite{bapirasi}:

\begin{Theorem} \label{theoprinc}
Let $q$ be a prime power and let $n>1$ be an integer. 
Let $F/\F_q$ be an algebraic function field of genus $g$ 
and $N_k$ be the number of places of degree $k$ in $F/\F_q$.
If $F/\F _q$ is such that $2g+1 \leq q^{\frac{n-1}{2}}(q^{\frac{1}{2}}-1)$ then:
\begin{enumerate}[1)]
	\item if $N_1 > 2n+2g-2$, then 
	$$ 
	\mus_q(n) \leq 2n+g-1,
	$$
	\item if ${N_1+2N_2>2n+2g-2}$ and there exists a non-special divisor of degree $g-1,$  then 
$$
\mus_q(n)\leq 3n+2g.
$$
	
\end{enumerate}
\end{Theorem}

\begin{Theorem}
  Let $q$ be a power of a prime $p$ and let $n$ be an integer. 
  Then the symmetric tensor rank $\mus_q(n)$ of
  multiplication in any finite field $\F_{q^n}$ is linear with respect to the
  extension degree; more precisely, there exists a constant $C_q$ such that for any integer $n>1$,
  \begin{equation*}
    \mus_q(n) \leq C_q n.
  \end{equation*}
\end{Theorem}

From different versions  of symmetric algorithms of Chudnovsky type  applied to good towers of algebraic function fields of type Garcia--Stichtenoth 
attaining the Drinfeld--Vladuts bounds of order one, two or four, different authors have obtained uniform bounds for the tensor rank of multiplication, 
namely general expressions for $C_q$, such as the following best
currently published estimates:

\begin{Theorem}\label{theo_arnaudupdate}
  Let ${q=p^r}$ be a power of a prime $p$ and let $n$ be an integer $>1$. Then:
  \begin{enumerate}[(i)]
   \item If ${q=2}$, then $\displaystyle{\mus_{q}(n) \leq 15.46n}$ (cf. \cite[Corollary 29]{bapi2} and \cite{ceoz})
   \item If ${q=3}$, then $\displaystyle{\mus_{q}(n) \leq  7.732 n}$ (cf. \cite[Corollary 29]{bapi2} and \cite{ceoz})
    \item If ${q\geq 4}$, then $\displaystyle{\mus_{q}(n) \leq 3 \left(1 +
      \frac{\frac{4}{3}p}{q-3+2(p-1)\frac{q}{q+1}} \right)n}$  (cf. \cite{bapirasi})
    \item  If $p\geq5$, then  $\displaystyle{\mus_{p}(n) \leq  3\left(1+
      \frac{8}{3p-5}\right)n}$  (cf. \cite{bapirasi})
      \item If ${q\geq 4}$, then  $\displaystyle{\mus_{q^2}(n) \leq 2 \left(1 +
      \frac{p}{q-3 + (p-1)\frac{q}{q+1}} \right)n}$  (cf. \cite{arna1} and \cite{bapirasi})
       \item If $p\geq5$, then  $\displaystyle{\mus_{p^2}(n) \leq 2 \left(1 +
      \frac{2}{p-\frac{33}{16}} \right)n}$  (cf. \cite{bapirasi})
  \end{enumerate}
\end{Theorem}

\subsection{New results}

The main goal of the paper is to improve the upper bounds for $\mus_{q}(n)$ from the previous theorem for the assertions concerning 
the extensions of finite fields $\F_{p^2}$ and $\F_{p}$ where $p$ is a prime number. Note that one of main ideas used in this paper was 
introduced in \cite{ball5} by the first author thanks to the use of the Chebyshev Theorem (or also called the Bertrand Postulat) to 
bound the gaps between prime numbers in order to construct families of modular curves as dense as possible. Later, motivated by \cite{ball5}, 
the approach of using such bounds on gaps between prime numbers (e.g. Baker-Harman-Pintz) was also used in the preprint \cite{Rand4} in order 
to improve the upper bounds of $\mus_{p^2}(n)$ where $p$ is a prime number. In our paper, we improve all the known uniform upper bounds 
for $\mus_{p^2}(n)$ and $\mus_{p}(n)$ for $p\geq 5$.

\section{New upper bounds}

In this section, we give new better upper bounds for the symmetric tensor rank of multiplication in certain extensions of 
finite fields $\F_{p^2}$ and $\F_{p}$. In order to do that, we construct suitable families of modular curves defined over $\F_{p^2}$ and $\F_{p}$.

\vspace{.5em}

\begin{Theorem}\label{lemmek0}
Let $l_k$ be the $k$-th prime number. Then there exists a real number $\alpha<1$ such that the difference between two consecutive prime numbers $l_k$ and $l_{k+1}$ satisfies 
$$l_{k+1}-l_k\leq l_k^{\alpha}$$ for any prime $l_k\geq x_{\alpha}.$ 

In particular, one can take $\alpha=\frac{21}{40}$ with the value of $x_{\alpha}$ that can in principle be determined effectively, or $\alpha=\frac{2}{3}$ with $x_{\alpha}=\exp(\exp(33.3)).$
\end{Theorem}

\begin{proof}
It is known that for all $x>x_\alpha$, the interval $[x-x^{\alpha},x]$ with $\alpha=\frac{21}{40}$ contains prime numbers by a result 
of Baker, Harman and Pintz \cite[Theorem 1]{bahapi}. Moreover, the value of $x_{\alpha}$ can in principle be determined, according to the authors. 
However, to our knowledge, this computation has not been realized yet. 

For a bigger $\alpha=\frac{2}{3}$, Dudek obtained recently in \cite{dude} an explicit bound $x_{\alpha}\geq\exp(\exp(33.3)).$

\qed
\end{proof}

\subsection{The case of the quadratic extensions of prime fields}\label{caseqe}

\begin{propo}\label{newboundpcarré}
Let $p\geq 5$ be a prime number, and let $x_{\alpha}$ be the constant from Theorem \ref{lemmek0}.

\begin{enumerate}
\item If $p\neq 11,$ then for any integer $\ds n\geq \frac{p-3}{2}x_\alpha+\frac{p+1}{2}$ we have

$$
\mus_{p^2}(n) \leq 2\left(1+\frac{1+\epsilon_p(n)}{p-3}\right)n-\frac{(1+\epsilon_p(n))(p+1)}{p-3} - 1,
$$
where $\ds\epsilon_p(n)=\left(\frac{2n}{p-3}\right)^{\alpha-1}.$

\item For $p=11$ and $\ds n\geq (p-3)x_\alpha+p-1=8x_\alpha+10$ we have
$$
\mus_{p^2}(n) \leq 2\left(1+\frac{1+\epsilon_p(n)}{p-3}\right)n-\frac{2(1+\epsilon_p(n))(p-1)}{p-3},
$$
where $\ds\epsilon_{p}(n)=\left(\frac{n}{p-3}\right)^{\alpha-1}.$

\item Asymptotically the following inequality holds for any $p\geq 5$:
$$
\Ms_{p^2} \leq 2\left(1+\frac{1}{p-3}\right).
$$
\end{enumerate}
\end{propo}

\begin{proof}

First, let us consider the characteristic $p$ such that $p\neq 11$. 
Then it is known (\cite[Corollary 4.1.21]{tsvl2} and \cite[proof of Theorem 3.9]{shtsvl})  
that the modular curve $X_k=X_0(11l_k)$, where $l_k$ is the $k$-th prime number, is of genus $g_k=l_k$ and satisfies 
${N_1(X_k(\F_{p^2}))\geq (p-1)(g_k+1)},$ where $N_1(X_k(\F_{p^2}))$ denotes the number 
of  rational points over $\F_{p^2}$ of the curve $X_k$.
Let us consider an integer $n>1$.
Then there exist two consecutive prime numbers $l_k$ and $l_{k+1}$ such 
that 
\begin{equation}
(p-1)(l_{k+1}+1)> 2n+2l_{k+1}-2
\end{equation} 
and 
\begin{equation}\label{Firstinequa2}
(p-1)(l_k+1)\leq 2n+2l_k-2
\end{equation}
(here we use the fact that $p\geq 5$). Let us consider  the algebraic function field 
$F_{k+1}/\F_{p^2}$ associated to the curve $X_{k+1}$ of genus $l_{k+1}$ defined over $\F_{p^2}$. 
Denoting by $N_i(F_{k}/\F_{p^2})$ the number 
of places of degree $i$ of $F_{k}/\F_{p^2}$, we get
$$
{N_1(F_{k+1}/\F_{p^2})\geq (p-1)(l_{k+1}+1)> 2n+2l_{k+1}-2}.
$$ 

We also know that $l_{k+1}-l_k\leq l_k^{\alpha},$  when $l_{k}\geq x_{\alpha}$ 
by Theorem \ref{lemmek0}. Thus $l_{k+1}\leq (1+\epsilon(l_k))l_k,$ with $\epsilon(l_k)= l_k^{\alpha-1}$.

It is easy to check that the inequality $2g+1\leq q^{\frac{n-1}{2}}(q^{\frac{1}{2}}-1)$ of Theorem \ref{theoprinc} holds for any prime power $q\geq 5.$ Indeed, it is enough to verify that
$$
q^{l_k\frac{p-3}{4}+\frac{p-1}{4}}(q^{\frac{1}{2}}-1)\geq 2(1+\epsilon(l_k))l_k+1,
$$ 
which is true since
$$
q^{x\frac{p-3}{4}+\frac{p-1}{4}}(q^{\frac{1}{2}}-1) -4x -1 \geq 0
$$
for any $x\geq 0.$

Thus, for any integer $n\geq\frac{p-3}{2}x_{\alpha}+\frac{p+1}{2}$ the function field $F_{k+1}/\F_{p^2}$ satisfies Theorem \ref{theoprinc}, so
$$
\mus_{p^2}(n)\leq 2n+l_{k+1}-1\leq 2n+(1+\epsilon(l_k))l_k-1,
$$ 
with $l_k\leq \frac{2n}{p-3}-\frac{p+1}{p-3}$  by \eqref{Firstinequa2}.

Let us remark that, as $l_k\leq \frac{2n}{p-3}$, $\epsilon(l_k)\leq \epsilon_p(n)=(\frac{2n}{p-3})^{\alpha-1},$ which gives the first inequality. 

Now, let us consider the characteristic $p=11$. Take the modular curve $X_k=X_0(23l_k)$, where $l_k$ is the $k$-th prime number. By \cite[Proposition 4.1.20]{tsvl2}, we easily compute that the genus of $X_k$ is $g_k=2l_k+1$. 
It is also known that the curve $X_k$ has good reduction modulo $p$ outside $23$ and $l_k$.
Moreover, by using \cite[Proof of Theorem 4.1.52]{tsvl2}, we obtain that the number of $\F_{p^2}$-rational points 
over  of the reduction  $X_k/p$ modulo $p$ satisfies 
$$
N_1(X_k(\F_{p^2}))\geq \frac{\mu_N(p-1)/12}{\deg \lambda_N} \geq 2(p-1)(l_k+1)
$$
in the notation of loc. cit.

Let us take an integer $n>1$. There exist two consecutive prime numbers $l_k$ and $l_{k+1}$ such 
that 
$$
{2(p-1)(l_{k+1}+1)> 2n+2(2l_{k+1}+1)-2}$$ and $${2(p-1)(l_k+1)\leq 2n+2(2l_k+1)-2},
$$ 
i.e.
\begin{equation}
(p-1)(l_{k+1}+1)> n+2l_{k+1} 
\end{equation}
and 
\begin{equation}\label{Secondinequa2}
(p-1)(l_k+1)\leq n+2l_k.
\end{equation}
Let us consider  the algebraic function field 
$F_{k+1}/\F_{p^2}$ associated to the curve $X_{k+1}$ of genus $g_{k+1}=2l_{k+1}+1$ defined over $\F_{p^2}$. We have
$$
N_1(F_{k+1}/\F_{p^2})\geq 2(p-1)(l_{k+1}+1)> 2n+4l_{k+1}.
$$ 
 
As before $l_{k+1}\leq (1+\epsilon(l_k))l_k,$ with $\epsilon(l_k)= l_k^{\alpha-1}$.

It is also easy to check that the inequality $2g+1\leq q^{\frac{n-1}{2}}(q^{\frac{1}{2}}-1)$ of Theorem \ref{theoprinc} holds when $q$ is a power of $11,$ which  follows from the fact that
$$
11^{4l_k+\frac{9}{2}}(11^{\frac{1}{2}}-1)\geq 8 l_k+3.
$$ 

Thus, for any integer $n\geq (p-3)x_{\alpha}+p-1$, 
the algebraic function field $F_{k+1}/\F_{p^2}$ satisfies Theorem \ref{theoprinc}, 
so 
$$
\mus_{p^2}(n)\leq 2n+2l_{k+1}\leq 2n+2(1+\epsilon(l_k))l_k
$$ 
with $l_k\leq \frac{n}{p-3}-\frac{p-1}{p-3}$  by \eqref{Secondinequa2}.  

We remark that as $l_k\leq \frac{n}{p-3}$, $\epsilon(l_k)\leq\epsilon_p(n)=(\frac{n}{p-3})^{\alpha-1},$ which gives the second inequality of the proposition.

Finally, when $n\rightarrow +\infty,$ the prime numbers $l_k\rightarrow +\infty,$ thus both for $p\neq 11$ and $p=11$ the corresponding $\epsilon_p(n) \rightarrow 0.$ So in the two cases
we obtain 
$$
M^{sym}_{p^2}\leq 2\left(1+\frac{1}{p-3}\right).
$$
\qed
\end{proof}

\begin{rema}
It is easy to see that the bounds obtained in Proposition \ref{newboundpcarré} are generally better than the published best known 
bounds (v) and (vi) recalled in Theorem \ref{theo_arnaudupdate}. Indeed, it is sufficient to consider the asymptotic bounds which are deduced from them and to see that 
 for any prime $p\geq 5$ we have $\frac{1}{p-3}<\frac{p}{p-3+(p-1)\frac{p}{p+1}}$ and  $\frac{1}{p-3}<\frac{2}{p-\frac{33}{16}}$ respectively.
\end{rema}

\begin{rema}
Note that the bounds obtained in \cite[Corollary 28]{Rand4} also concern the symmetric tensor rank of multiplication in the finite fields even if it is not mentioned.
Indeed, the distinction between $\mus_{q}(n)$ and $\mu_{q}(n)$ was exploited only from \cite{Rand3}. So, we can compare our proposition \ref{newboundpcarré} 
with Corollary 8 there. Firstly, note that the bounds in \cite[Corollary 28]{Rand4} are only valid for $p\geq 7$. Moreover, the only bound which is best than our bounds 
is the asymptotic bound \cite[Corollary 28, Bound (vi)]{Rand4} given for  an unknown sufficiently large $n$, contrary to our uniform bound with $\alpha=\frac{2}{3}$ for $n\geq exp(\exp(33.3))$. 
\end{rema}

\subsection{The case of prime fields}

\begin{propo}\label{newboundp}
Let $p\geq 5$ be a prime number, let $x_{\alpha}$ be defined as in Lemma \ref{lemmek0}, and $\epsilon_p(n)$ as in Proposition \ref{newboundpcarré}.

\begin{enumerate}
\item If $p\neq 11,$ then for any integer $\ds n\geq \frac{p-3}{2}x_\alpha+\frac{p+1}{2}$ we have

$$
\mus_{p}(n) \leq 3\left(1+\frac{\frac{4}{3}(1+\epsilon_p(n))}{p-3}\right)n-\frac{2(1+\epsilon_p(n))(p+1)}{p-3}.
$$

\item For $p=11$ and $\ds n\geq (p-3)x_\alpha+p-1=8x_\alpha+10$ we have
$$
\mus_{p}(n) \leq 3\left(1+\frac{\frac{4}{3}(1+\epsilon_p(n))}{p-3}\right)n-\frac{4(1+\epsilon_p(n))(p-1)}{p-3}+1.
$$

\item Asymptotically the following inequality holds for any $p\geq 5$:
$$
\Ms_{p} \leq 3\left(1+\frac{\frac{4}{3}}{p-3}\right).
$$
\end{enumerate}
\end{propo}

\begin{proof}
It suffices to consider the same families of curves as in the proof of Proposition \ref{newboundpcarré}. 

When $p\neq 11$ we take $X_k=X_0(11l_k)$, where $l_k$ is the $k$-th prime number. These curves are defined over $\F_p$, hence, we can consider the associated algebraic function fields $F_k/\F_p$ defined over $\F_p$ and we have $N_1(F_{k}/\F_{p^2})=N_1(F_{k}/\F_{p})+2N_2(F_{k}/\F_{p})\geq (p-1)(l_{k}+1),$ 
since $F_{k}/\F_{p^2}=F_{k}/\F_{p}\otimes_{\F_p} \F_{p^2}$ for any $k$. 
Note that the genus of the algebraic function fields $F_k/\F_p$ is also $g_k=l_k,$ 
since the genus is preserved under descent.

Given an integer $n>1$, there exist two consecutive prime numbers $l_k$ and $l_{k+1}$ such 
that 
\begin{equation}
(p-1)(l_{k+1}+1)> 2n+2l_{k+1}-2
\end{equation} 
and 
\begin{equation}\label{Thirdinequa2}
(p-1)(l_k+1)\leq 2n+2l_k-2. 
\end{equation}

Let us consider  the algebraic function field 
$F_{k+1}/\F_{p}$ associated to the curve $X_{k+1}$ of genus $l_{k+1}$ defined over $\F_{p}$. We get
$$
N_1(F_{k+1}/\F_{p})+2N_2(F_{k+1}/\F_{p})\geq (p-1)(l_{k+1}+1)> 2n+2l_{k+1}-2.
$$ 

As before $l_{k+1}\leq (1+\epsilon(l_k))l_k,$ with $\epsilon(l_k)= l_k^{\alpha-1},$ and from the proof of the previous proposition we know that the inequality $2g+1\leq q^{\frac{n-1}{2}}(q^{\frac{1}{2}}-1)$ of Theorem \ref{theoprinc} holds. Consequently, for any integer $n\geq \frac{p-3}{2}x_{\alpha}+\frac{p+1}{2}$, the algebraic function field $F_{k+1}/\F_{p}$ satisfies Theorem \ref{theoprinc}, 2) since by \cite[Theorem 11 (i)]{balb} 
there always exists a non-special divisor of degree $g_{k+1}-1$ for $p\geq 5$.
So 
$$
\mus_{p}(n)\leq 3n+2l_{k+1}\leq 3n+2(1+\epsilon(l_k))l_k
$$ 
with $l_k\leq \frac{2n}{p-3}-\frac{p+1}{p-3}$ by \eqref{Thirdinequa2}. As before, $\epsilon(l_k)\leq\epsilon_p(n)=(\frac{2n}{p-3})^{\alpha-1}.$

When $p=11$ we use once again the family of curves $X_k=X_0(23l_k).$ They are defined over $\F_p,$ hence we can consider the associated algebraic function fields $F_k/\F_p$ over $\F_p$ 
and we have $N_1(F_{k}/\F_{p^2})=N_1(F_{k}/\F_{p})+2N_2(F_{k}/\F_{p})\geq (p-1)(l_{k}+1).$ The genus of the algebraic function fields $F_k/\F_p$ defined over $\F_p$ is also $g_k=2l_k+1$ 
since the genus is preserved under descent.

Given an integer $n>1$, there exist two consecutive prime numbers $l_k$ and $l_{k+1}$ such 
that 
$$
2(p-1)(l_{k+1}+1)> 2n+2(2l_{k+1}+1)-2
$$ 
and 
$$
2(p-1)(l_k+1)\leq 2n+2(2l_k+1)-2, 
$$
i.e. 
\begin{equation}
(p-1)(l_{k+1}+1)> n+2l_{k+1} 
\end{equation}
and 
\begin{equation}\label{Fourinequa2}
(p-1)(l_k+1)\leq n+2l_k.
\end{equation}
 
Let us consider  the algebraic function field 
$F_{k+1}/\F_{p}$ associated to the curve $X_{k+1}$ of genus $g_{k+1}=2l_{k+1}+1$ defined over $\F_{p}$. 
We get
$$
N_1(F_{k+1}/\F_{p})+2N_2(F_{k+1}/\F_{p})\geq 2(p-1)(l_{k+1}+1)> 2n+2(2l_{k+1}+1)-2.
$$ 

As above $l_{k+1}\leq (1+\epsilon(l_k))l_k,$ with $\epsilon(l_k)= l_k^{\alpha-1},$ and the inequality $2g+1\leq q^{\frac{n-1}{2}}(q^{\frac{1}{2}}-1)$ of Theorem \ref{theoprinc} holds.  Consequently, for any integer $n\geq (p-3)x_{\alpha}+ p-1$, the algebraic function field $F_{k+1}/\F_{p}$ satisfies Theorem \ref{theoprinc}, 2) since, as before, there exists a non-special divisor of degree $g_{k+1}-1$ by \cite[Theorem 11 (i)]{balb}. So,
$$
\mus_{p}(n)\leq 3n+2g_{k+1}\leq 3n+2(2l_{k+1}+1)\leq 3n+2(1+\epsilon)l_k
$$ 
with $l_k\leq \frac{n}{p-3}-\frac{p-1}{p-3}$ by \eqref{Fourinequa2}. We can also bound $\epsilon(l_k)\leq\epsilon_p(n)=(\frac{n}{p-3})^{\alpha-1}.$

Finally, when $n\rightarrow +\infty,$ the prime numbers $l_k\rightarrow +\infty,$ thus both for $p\neq 11$ and $p=11,$ $\epsilon_p(n) \rightarrow 0.$ So 
we obtain $M^{sym}_{p}\leq 3\left(1+\frac{\frac{4}{3}}{p-3}\right)$. \qed

\end{proof}

\begin{rema}
It is easy to see that the bounds obtained in Proposition \ref{newboundp} are generally better than the best known 
bounds (iii) and (iv) recalled in Theorem \ref{theo_arnaudupdate}. Indeed, it is sufficient to consider the asymptotic bounds which are deduced from them and to see that 
for any prime $p\geq 5$ we have $\frac{\frac{4}{3}}{p-3}<\frac{\frac{4}{3}p}{p-3+\frac{2(p-1)p}{p+1}}$ and  $\frac{\frac{4}{3}}{p-3}<\frac{8}{3p-5}$ respectively.

\end{rema}

\section*{Acknowledgments}
 The first author wishes to thank Sary Drappeau, Olivier Ramaré, Hugues Randriambololona, Joël Rivat and Serge Vladuts for valuable discussions.

\end{document}